\documentclass{amsart}   

\usepackage{amssymb, hyperref, multicol}
\usepackage{amsthm, amsmath,amssymb,amsbsy,amsfonts, latexsym,amsopn,amstext, amsxtra,euscript,amscd}
\usepackage{array, caption, floatrow, tabularx, makecell, booktabs}%
\usepackage{mathabx}

\setcellgapes{3pt}

\usepackage[usenames,dvipsnames]{color}

\usepackage[shortlabels]{enumitem}

\usepackage{hyperref}

\hypersetup{
  colorlinks   = true, 
  urlcolor     = blue, 
  linkcolor    = blue, 
  citecolor   = red 
}

\usepackage[sorted, msc-links]{amsrefs}
 
\usepackage[capitalise]{cleveref}
\usepackage{aliascnt}

\newtheorem{thm}{Theorem}
\newtheorem {prop}{Proposition}
\newtheorem {lem}{Lemma}
\newtheorem {cor}{Corollary}
\newtheorem{claim}{Claim}

\theoremstyle{definition}

\newtheorem{rem}{Remark}

\DeclareMathOperator\Aut{Aut \, }
\DeclareMathOperator\Gal{Gal }

\DeclareMathOperator\PSL{PSL }

\newcommand\X{\mathcal X}            
\newcommand\G{\overline{G}}
\newcommand\N{\overline{N}}

\newcommand\C{\mathbb C}

\begin{document}

\title{Generalized superelliptic Riemann surfaces}

\author{Rub\'en A. Hidalgo}
\address{Departamento de Matem\'atica y Estad\'{\i}stica, Universidad de La Frontera, Temuco, Chile.}
\email{ruben.hidalgo@ufrontera.cl, saul.quispe@ufrontera.cl}

\author{Sa\'ul Quispe}

\thanks{The first two authors were partially supported by Projects Fondecyt 1230001 and 1220261}

\author{Tony Shaska}
\address{Department of Mathematics, Oakland University, Rochester, MI, 48386. }
\email{shaska@oakland.edu}

\begin{abstract}
A conformal automorphism $\tau$, of order $n \geq 2$, of a closed Riemann surface $\X$, of genus $g \geq 2$, which is central in ${\rm Aut}(\X)$ and such that $\X/\langle \tau\rangle$ has genus zero, is  called a superelliptic automorphism of level $n$. If $n=2$, then $\tau$ is called hyperelliptic involution and it is known to be unique. In this paper, for the case $n \geq 3$, we investigate the uniqueness of the cyclic group $\langle \tau \rangle$.
Let $\tau_{1}$ and $\tau_{2}$ be two superelliptic automorphisms of level $n$ of $\X$.  If $n \geq 3$ is odd, then $\langle \tau_{1} \rangle =\langle \tau_{2} \rangle$.
In the case that $n \geq 2$ is even, then the same uniqueness result holds, up to some explicit exceptional cases. 
\end{abstract}


\maketitle

\section{Introduction} 
Let  $\X$ be a closed Riemann surface of genus $g \geq 2$ and let $G=\Aut (\X)$ be its group of conformal automorphisms. It is well known that ${\rm Aut}(\X)$ is finite \cite{Schwarz} of order at most $84(g-1)$ \cite{Hurwitz}, and that the order of any conformal automorphism is bounded above by $4g+2$. This paper considers certain cyclic subgroups of ${\rm Aut}(\X)$ which behave similarly to the hyperelliptic involution.

Let $\tau \in G$ be an {\it $n$-gonal} automorphism, that is, it has order $n \geq 2$ and $\X/\langle \tau \rangle$ has genus zero. In this case, $H=\langle \tau \rangle \cong C_{n}$ is called an {\it $n$-gonal group} and $\X$ a {\it cyclic $n$-gonal} Riemann surface. Let $N$ be the normalizer of $\langle \tau \rangle$ in $G$. It follows from the results in \cites{Singerman, Greenberg}, that generically $N=G$.

If $n=2$, then $\tau$ is the hyperelliptic involution and it is known to be unique and central in $G$; in particular, it is central in $G=N$. 

If $n \geq 3$ is a prime integer and $s \geq 3$ is the number of fixed points of $\tau$, then $H$ is known to be the unique $n$-Sylow subgroup of $G$ if either (i) $2n<s$ \cite{Severi} or (ii) $n \geq 5s-7$ \cite{Hidalgo:pgrupos}. So, in this case, $N=G$; but it might be that $\tau$ is non-central.

If $n \geq 3$, not necessarily prime, such that: (i) every fixed point of a non-trivial power of $\tau$ is also a fixed point of $\tau$, and (ii) the rotation number of $\tau$ at each of its fixed points is the same, (some authors call $\tau$ a {\it superelliptic automorphism} and $\X$ a {\it superelliptic surface}), then $\tau$ is central in $N$ (see Corollary \ref{corosupereliptico}), but in general $N \neq G$. In this case, under the extra condition that $g>(n-1)^2$, it is known that $N=G$ \cite{K} (as a consequence of results in \cite{Accola1}). The computation of $G$ has been done in \cite{Sanjeewa}. Superelliptic Riemann surfaces have been studied in \cites{BHS, m-sh, MPRZ} and those with many conformal automorphisms and with CM structures have been considered in \cite{obus-sh}.

If $\tau$ is central in $G$ (respectively, central in $N$), then we call it a \textbf{superelliptic automorphism of level ${\bf n}$} (respectively, \textbf{generalized superelliptic automorphism of level ${\bf n}$}); we also say that $H=\langle \tau \rangle$ is a  \textbf{superelliptic group of level ${\bf n}$} (respectively, \textbf{generalized superelliptic group of level ${\bf n}$}), and that $\X$ is a \textbf{superelliptic curve of level ${\bf n}$} (respectively, \textbf{generalized superelliptic curve of level ${\bf n}$}). A superelliptic automorphism of level $n$ is automatically a generalized one; the converse is in general false (but generically true). Also, as previously noted, a superelliptic automorphism of order $n$ is a generalized superelliptic of level $n$ (but it might not be a superelliptic automorphism of level $n$).

In this paper, (i) we provide necessary and sufficient conditions for an $n$-gonal automorphism to be a generalized superelliptic automorphism of level $n$ (Theorem \ref{exponentes}) and (ii) we provide conditions for a superelliptic curve of level $n$ to have a unique superelliptic group of level $n$ (Theorem \ref{unicidad0} and Corollary \ref{unicidad}).

Before stating the above two results, let us recall some facts on $n$-gonal automorphisms.
Let us consider a pair $(\X,\tau)$, where $\tau$ is a $n$-gonal automorphism of $\X$. Set $H=\langle \tau \rangle \cong C_{n}$.
Let $\pi:\X \to \widehat \C$ be a Galois branched covering, whose deck covering group is $H=\langle \tau \rangle$, and let $p_{1},\ldots,p_{s} \in \widehat{\C}$ be its branch values. Then there are integers $l_{1},\ldots,l_{s} \in \{1,\ldots,n-1\}$ satisfying that $l_{1}+\cdots+l_{s}$ is a multiple of $n$ and $\gcd(n,l_{1},\ldots,l_{s})=1$, such that 
$\X$ can be described by an affine irreducible algebraic curve (which might have singularities) of the following form (called a \textbf{cyclic $n$-gonal curve})
\begin{equation}\label{ngonal}
\begin{array}{l}
y^{n}=\prod_{j=1}^{s}(x-p_{j})^{l_{j}}.
\end{array}
\end{equation}

If one of the branch values is $\infty$, say $p_{s}=\infty$, then we need to delete the factor $(x-p_{s})^{l_{s}}$ from the above equation.
In this algebraic model, $\tau$ and $\pi$ are given respectively by $\tau(x,y)=(x,\omega_{n}y)$, where $\omega_{n}=e^{2 \pi i/n}$, and $\pi(x,y)=x$.

\begin{thm}\label{exponentes}
Let $\X$ be a cyclic $n$-gonal Riemann surface, described by the cyclic $n$-gonal curve  \cref{ngonal}, 
and $N$ be the normalizer of $H=\langle \tau(x,y)=(x,\omega_{n}y) \rangle$ in $\Aut(\X)$. 
Let $\theta:N \to \overline{N}=N/H$ be the canonical projection homomorphism. Then 
$\tau$ is a generalized superelliptic automorphism of level $n$ if and only if for all $p_{j}$ and $p_{i}$ in the same $\theta(N)$-orbit it holds that $l_{j}=l_{i}$. 
\end{thm}

\begin{cor}\label{corosupereliptico}
Let $\X$ be a cyclic $n$-gonal Riemann surface, described by the cyclic $n$-gonal curve  \cref{ngonal}. 
If $l_{j}=l$, for every $j$, where $\gcd(n,l)=1$, then $\tau(x,y)=(x,\omega_{n}y) \in \Aut (\X)$ is a generalized superelliptic automorphism of level $n$.
\end{cor}

\begin{rem}
The above corollary states that a superelliptic automorphism is always a generalized superelliptic automorphism of level $n$ (but not necessarily a superelliptic automorphism of level $n$ as the normalizer $N$ might be smaller than the full group of automorphisms).
\end{rem}

\begin{thm}\label{unicidad0}
Let $\X$ be a cyclic $n$-gonal Riemann surface, admitting two superelliptic automorphisms $\tau$ and $\eta$, both of level $n$, such that 
$\langle \tau\rangle \neq \langle \eta \rangle$.  Then
\begin{enumerate}
\item[(I)] ${\rm Aut}(\X)/\langle \tau \rangle$ is either a non-trivial cyclic group of even order or a dihedral group of order a multiple of four;
\item[(II)] there is an integer $d \geq 2$ such that 
$n=2d$ and $\X$ can be represented by a cyclic $n$-gonal curve of the form
\begin{equation}\label{eq0}
\begin{array}{l}
\X: \quad y^{2d}=x^{2}\left(x^{2}-1\right)^{l_{1}}\left(x^{2}-a_{2}^{2}\right)^{l_{2}} \prod_{j=3}^{L}\left(x^{2}-a_{j}^{2}\right)^{2\widehat{l}_{j}},
\end{array}
\end{equation}
where 
$l_{1},l_{2},2\widehat{l}_{3},\ldots,2\widehat{l}_{L} \in \{1,\ldots,2d-1\}$, $l_{1}$ is odd, and 
either one of the two conditions (a) or (b) below holds for $l_{2}$.
\begin{enumerate}
\item[(a)] If $l_{2}=2\widehat{l}_{2}$, then $\gcd\left(d,l_{1},\widehat{l}_{2},\ldots,\widehat{l}_{L}\right)=1$.
\item[(b)] If  $l_{2}$ is odd, then $\gcd\left(d,l_{1},l_{2},\widehat{l}_{3},\ldots,\widehat{l}_{L}\right)=1$.
\end{enumerate}
In these cases, $\tau(x,y)=(x,\omega_{2d}y)$, $\eta(x,y)=(-x,\omega_{2d}y)$ and $\langle \tau,\eta\rangle \cong C_{2d} \times C_{2}$.

\end{enumerate}

\end{thm}

Those superelliptic Riemann surfaces of level $n=2d$, described by the cyclic $2d$-gonal curves in \cref{unicidad0}, will be called {\bf exceptionals}.

\begin{rem}
The cyclic $2d$-gonal curves $\X$, defined by \cref{eq0}, are cyclic $2d$-gonal curves $\X$ admitting two commuting cyclic $2d$-gonal automorphisms, $\tau, \eta$, such that $\langle \tau \rangle \neq \langle \eta \rangle$.
We should note that not all of them need to be hyperelliptic of level $n$; the theorem only asserts that the exceptional ones are some of them. For instance, in the case (b) with $d=2$, $l_{1}=1$ and $l_{2}=3$, the genus five curve 
$\X: \quad y^{4}=x^{2}(x^{2}-1)(x^{2}+1)^3$
admits the automorphism 
$\rho(x,y)=\left(ix, \frac{\sqrt{i} y^{3}}{x(x^{2}+1)^{2}}\right),$
for which $\rho \tau \rho^{-1}=\tau^{3}$.
\end{rem}

\begin{cor}\label{unicidad}
Let $\X$ be a Riemann surface admitting a superelliptic group $H$ of level $n$. Then $H$ is the unique superelliptic group of level $n$ of $\X$ if 
either: (1) $n=2$, or (2) $n \geq 3$ is odd , or (3) $n \geq 4$ is even and $\X/H$ has no cone point of order $n/2$.
\end{cor}

Finally, in the last section, we provide some discussion on the field of moduli of these superelliptic Riemann surfaces (see Theorem \ref{teounico}).

\medskip
\noindent  
{\bf Notations.}
We denote by $C_{n}$ the cyclic group of order $n$, by $D_{n}$ the dihedral group of order $2n$, by $A_n$ the alternating group,  and by $S_n$ the symmetric group. 

\section{Preliminaries}
\subsection{The finite groups of M\"obius transformations}\label{Sec:gruposfinitos}
Up to $\PSL_{2}(\C)$-conjugation, the finite subgroups of the group $\PSL_{2}(\C)$ of M\"obius transformations are given by (see, for instance, \cite{Beardon})
{\small
\begin{equation}
\begin{array}{lll}
C_{m}:=\big\langle a(x)=\omega_{m} x \big\rangle, \;
D_{m} :=\Big\langle a(x)=\omega_{m} x, b(x)=\frac{1}{x} \Big\rangle,\;
A_{4}:=\Big\langle a(x)=-x, b(x)=\frac{i-x}{i+x}\Big\rangle,\\
S_{4}:=\Big\langle a(x)=ix, b(x)=\frac{i-x}{i+x}\Big\rangle,\;
A_{5}  :=\Big\langle a(x)=\omega_{5} x, b(x)=\frac{(1-\omega_{5}^{4})x+(\omega_{5}^{4}-\omega_{5})}{(\omega_{5}-\omega_{5}^{3})x+(\omega_{5}^{2}-\omega_{5}^{3})}\Big\rangle,
\end{array}
\end{equation}
}\noindent
where $\omega_m$ is a primitive $m$-th root of unity. For each of the above finite groups $A$, 
a Galois branched covering $f_{A}:\widehat\C \to \widehat\C$, with deck group $A$, is given as follows 
\[
\begin{array}{lcll}
f_{C_{m}}(x)&=&x^{m}; & \mbox{branching: } \; (m,m).\\
f_{D_{m}}(x)&=&x^{m}+x^{-m}; & \mbox{branching: } \; (2,2,m).\\
f_{A_{4}}(x)&=&\dfrac{(x^{4}-2i\sqrt{3} x^{2}+1)^{3}}{-12i\sqrt{3} x^{2}(x^{4}-1)^{2}}; & \mbox{branching: } \; (2,3,3).\\
f_{S_{4}}(x)&=&\dfrac{(x^{8}+14x^{4}+1)^{3}}{108x^{4}(x^{4}-1)^{4}}; & \mbox{branching: } \; (2,3,4).\\
f_{A_{5}}(x)&=&\dfrac{(-x^{20}+228x^{15}-494x^{10}-228x^{5}-1)^{3}}{1728x^{5}(x^{10}+11x^{5}-1)^{5}}; & \mbox{branching: } \; (2,3,5),
\end{array}
\]
see \cite{Horiuchi}. 
In the above, the branching corresponds to the tuple of branch orders of the cone points of the orbifold $\widehat\C/A$.

\subsection{Fuchsian groups}
A \emph{Fuchsian group} is a discrete subgroup $\Delta$ of $\PSL_{2}({\mathbb R})$, the group orientation-preserving isometries of the hyperbolic plane ${\mathbb H}$. It is called co-compact if the quotient orbifold ${\mathbb H}/\Delta$ is compact; its signature is the tuple  $(g; n_{1},\ldots,n_{s})$,
where $g$ is the genus of the quotient orbifold ${\mathbb H}/\Delta$, $s$ is the number of its cone points they having branch orders $n_{1},\ldots, n_{s}$. The group $\Delta$ has a presentation as follows:
\begin{equation}\label{eq2}
\begin{array}{l}
\Gamma=\langle a_{1},b_{1},\ldots,a_{g}, b_{g},c_{1},\ldots,c_{s}:
 c_{1}^{n_{1}}=\cdots=c_{s}^{n_{s}}=1, c_{1}\cdots c_{s}[a_{1},b_{1}] \cdots [a_{g},b_{g}]=1\rangle,
\end{array}
\end{equation}
where $[a,b]=aba^{-1}b^{-1}$.

If a co-compact Fuchsian group $\Gamma$ has no torsion, then $\X={\mathbb H}/\Gamma$ is a closed Riemann surface of genus $g \geq 2$ and its signature is $(g;-)$. Conversely, by the uniformization theorem, every closed Riemann surface of genus $g \geq 2$ can be represented as above. In this case, by Riemann's existence theorem, 
 a finite group $G$ acts faithfully as a group of conformal automorphisms of $\X$ if and only if there is a co-compact Fuchsian group $\Delta$ and a surjective homomorphism $\theta:\Delta \to G$ whose kernel is $\Gamma$.

\subsection{Cyclic $n$-gonal Riemann surfaces}
Let $\X$ be a cyclic $n$-gonal Riemann surface of genus $g \geq 2$,  $\tau \in \Aut (\X)$ be a $n$-gonal automorphism and  $\pi: \X \to \widehat\C$ be a Galois branched cover whose deck group is the $n$-gonal group $H=\langle \tau \rangle \cong C_{n}$. Let $p_{1},\ldots, p_{s} \in \widehat\C$ be the branch values of $\pi$ and let us denote  by $n_{j} \geq 2$ (which is a divisor of $n$) the branch order of $\pi$ at $p_{j}$.

Let $\Delta$ be a Fuchsian group such that (up to biholomorphisms) ${\mathbb H}/\Delta=\X/\langle \tau \rangle$. Then $\Delta$ has 
signature $(0;n_{1},\ldots,n_{s})$ and a presentation 
\begin{equation}
\begin{array}{l}
\Delta=\langle c_{1},\ldots,c_{s}:  c_{1}^{n_{1}}=\cdots=c_{s}^{n_{s}}=1, c_{1}\cdots c_{s}=1\rangle.
\end{array}
\end{equation}

The branched Galois covering $\pi$ is determined by a surjective homomorphism $\rho:\Delta \to C_{n}=\langle \tau \rangle$ with a torsion-free kernel $\Gamma$ such that $\X={\mathbb H}/\Gamma$. (The homomorphism $\rho$ is uniquely determined up to post-composition by automorphisms of $C_{n}$ and pre-composition by an automorphism of $K$.)
Let $\rho(c_{j})=\tau^{l_{j}}$, where $c_{j}$ is as in \cref{eq2}, for $l_{1},\ldots, l_{s} \in \{1,\ldots,n-1\}$. 
As a consequence of Harvey's criterion \cite{Harvey}, 
\begin{enumerate}
\item[(a)] $n={\rm lcm}(n_{1},\ldots , n_{j-1}, n_{j+1},\ldots , n_{s})$ for all $j$;
\item[(b)] if $n$ is even, then $\#\{ j \in \{1,\ldots,s\}: n/n_{j} \; \mbox{is odd} \}$ is even.
\end{enumerate}

The equality $c_{1}\cdots c_{s}=1$ is equivalent to have 
 $l_{1}+\cdots+l_{s} \equiv 0 \; {\rm mod}(n)$, and 
the condition for $\Gamma=\ker(\rho)$ to be torsion-free is equivalent to have 
$\gcd(n,l_{j})=n/n_{j}$, for $j=1,\ldots, s$. The surjectivity of $\rho$ is equivalent to have $\gcd(n,l_{1},\ldots,l_{s})=1$, which in our case is equivalent to condition (a).
Condition (b) is equivalent to saying that for $n$ even the number of $l_{j}$'s being odd is even, which trivially holds.
Summarizing all the above,  
\begin{multicols}{2}
\begin{enumerate}
\item $l_{1},\ldots, l_{s} \in \{1,\ldots,n-1\}$,
\item $l_{1}+\cdots+l_{s} \equiv 0 \;{\rm mod}(n)$,
\item $\gcd(n,l_{j})=n/n_{j}$, for all $j$,
\item $\gcd(n,l_{1},\ldots,l_{s})=1$.
\end{enumerate}
\end{multicols}

The Riemann surface $\X$ can be described by the affine curve
\begin{equation}\label{ecuacion}
\begin{array}{l}
 \X: \quad y^{n}=\prod_{j=1}^{s} (x-p_{j})^{l_{j}},
 \end{array}
 \end{equation}
where,  
if one of the branched values is infinity, say $p_{s}=\infty$, then we need to delete the factor $(x-p_{s})^{l_{s}}$ in the above equation.

In such an algebraic model, $\tau(x,y)=(x,\omega_{n} y)$, where $\omega_{n}=e^{2 \pi i/n}$, and  $\pi(x,y)=x$. The branch order of $\pi$ at $p_{j}$ is $n_{j}=n/\gcd(n,l_{j})$ and, by the Riemann-Hurwitz formula, the genus $g$ of $\X$ is given by
\begin{equation}
\begin{array}{l}
g=1+\frac{1}{2}\left((s-2)n-\sum_{j=1}^{s} \gcd(n,l_{j})\right).
\end{array}
\end{equation}

\section{Proof of Theorem \ref{exponentes}}\label{Sec:GSRS}
Let $\X$ be a curve given by equation \cref{ngonal}, $\pi(x,y)=x$, and $\tau \in G=\Aut (\X)$.
Let $N$ be the normalizer of $H=\langle \tau \rangle$ in $G$. There is a short exact sequence 
\begin{equation}
\begin{array}{l}
1 \to H=\langle \tau \rangle \to N \stackrel{\theta}{\to} \overline{N}=N/H \to 1,
\end{array}
\end{equation}
where $\theta(\eta) \circ \pi = \pi \circ \eta$, for every $\eta \in N$. 

The reduced group of automorphisms
$\overline{N}=N/H < \PSL_{2}(\C)$ is a finite group keeping invariant the set $\{p_{1},\ldots,p_{s}\}$. 

\subsection{}
The following describes the form of those elements of $N$.

\begin{lem}\label{lemaconmuta}
Let $\eta \in N$ and  $l \in \{1,\ldots,n-1\}$ (necessarily relatively prime to $n$) such that $\eta \tau \eta^{-1}=\tau^{l}$. If 
$b=\theta(\eta)$, then $\eta(x,y)=(b(x),y^{l}Q(x))$, where $Q(x)$ is a suitable rational map.
\end{lem}
\begin{proof}
Let us note that $\eta(x,y)=(b(x),L(x,y))$, where $L(x,y)$ is a suitable rational map.
As $\eta(\tau(x,y))=\eta(x,\omega_{n}y)=(b(x),L(x,\omega_{n}y))$ and 
$\tau^{l}(\eta(x,y))=\tau^{l}(b(x),L(x,y))=(b(x),\omega_{n}^{l}L(x,y))$, the condition $\eta \tau \eta^{-1}=\tau^{l}$ holds if and only if 
$L(x,\omega_{n}y)=\omega_{n}^{l}L(x,y)$, that is, $L(x,y)=Q(x)y^{l}$, for a suitable rational map $Q(x) \in \C(x)$. 
\end{proof}

\begin{rem}
(1) Lemma \ref{lemaconmuta} asserts that those $\eta \in N$ commuting with $\tau$ have the form $\eta(x,y)=(b(x),Q(x)y)$.
(2) If $t \in \PSL_{2}(\C)$, then replacing $\pi$ by $t \circ \pi$ only exchanges the set of branch points $\{p_{1},\ldots,p_{s}\}$ for $\{t(p_{1}),\ldots,t(p_{s})\}$ but keeps invariant the set of exponents $l_{1},\ldots, l_{s}$.
\end{rem}

\subsection{}
Let $\eta \in N$ and assume $\theta(\eta)$ has order $m \geq 2$. As there is a suitable $t \in \PSL_{2}(\C)$ such that $t \theta(\eta) t^{-1}(x)=\omega_{m} x$, we may assume (by post-composing $\pi$ with $t$) that $\theta(\eta)(x)=\omega_{m}x$. So the cyclic $n$-gonal curve \cref{ecuacion} can be written as
\begin{equation}\label{equcentral}
\begin{array}{l}
y^{n}=x^{\alpha}\prod_{j=1}^{L} (x-q_{j})^{l_{j,1}}(x-\omega_{m} q_{j})^{l_{j,2}} \cdots (x-\omega_{m}^{m-1}q_{j})^{l_{j,m}},
\end{array}
\end{equation}
where
\begin{enumerate}
\item[(a)] the factor $x^{\alpha}$ only appears if one of the branch values $t(p_{j})$ is equal to zero, and 
\item[(b)] there is the following equality of sets of exponents in \cref{equcentral} and \cref{ecuacion} (where $\alpha$ needs to be deleted if none of the $t(p_{j})$'s is equal to zero)
\[
\{\alpha,l_{1,1},\ldots,l_{1,m},l_{2,1},\ldots,l_{2,m},\ldots,l_{L,1},\ldots,l_{L,m}\}=\{l_{1},\ldots,l_{s}\}.
\]
\end{enumerate}

In this model, $\tau(x,y)=(x,\omega_{n}y)$ and, by \cref{lemaconmuta}, $\eta(x,y)=(\omega_{m}x,Q(x)y^{l})$, for a suitable rational map $Q(x) \in \C(x)$.  

If $R(x)$ denotes the right side of \cref{equcentral}, then $Q(x)^{n}y^{ln}=R(\omega_{m}x)$ on $\X$, where 
\begin{equation}
\begin{array}{l}
R(\omega_{m}x)=\omega_{m}^{\alpha} x^{\alpha}   \prod_{j=1}^{L}  
\frac{ \omega_{m}^{r_{j}} (x-q_{j})^{l_{j,1}}(x-\omega_{m}q_{j})^{l_{j,2}} \cdots (x-\omega_{m}^{m-1}q_{j})^{l_{j,m}}}
{(x-q_{j})^{l_{j,1}-l_{j,2}}(x-\omega_{m}q_{j})^{l_{j,2}-l_{j,3}} \cdots (x-\omega_{m}^{m-1}q_{j})^{l_{j,m}-l_{j,1}}}\\
\\
=\frac{\omega_{m}^{(\alpha+\sum_{j=1}^{L}r_{j})}y^{n}}{
\prod_{j=1}^{L}  
(x-q_{j})^{l_{j,1}-l_{j,2}}(x-\omega_{m} q_{j})^{l_{j,2}-l_{j,3}} \cdots (x-\omega_{m}^{m-1} q_{j})^{l_{j,m}-l_{j,1}}}, 
\end{array}
\end{equation}
and $r_{j}=l_{j,1}+\cdots+l_{j,m}$, that is,
\begin{equation}
\begin{array}{l}
Q(x)^{n}y^{ln}=\frac{\omega_{m}^{(\alpha+\sum_{j=1}^{L}r_{j})}y^{n}}{
\prod_{j=1}^{L}  
(x-q_{j})^{l_{j,1}-l_{j,2}}(x-\omega_{m} q_{j})^{l_{j,2}-l_{j,3}} \cdots (x-\omega_{m}^{m-1} q_{j})^{l_{j,m}-l_{j,1}}}.
\end{array}
\end{equation}

In particular, 
\begin{equation}\label{19}
\begin{array}{l}
Q(x)^{n}y^{(l-1)n}=\frac{\omega_{m}^{(\alpha+\sum_{j=1}^{L}r_{j})}}{
\prod_{j=1}^{L}  
(x-q_{j})^{l_{j,1}-l_{j,2}}(x-\omega_{m}q_{j})^{l_{j,2}-l_{j,3}} \cdots (x-\omega_{m}^{m-1}q_{j})^{l_{j,m}-l_{j,1}}}.
\end{array}
\end{equation}

\subsection{}
Let us assume $\eta$ commutes with $\tau$, that is, $l=1$. We proceed to prove that the exponents $l_{j,i}$ are the same for every $i=1,\ldots,m$.
As $\theta(\eta)^{m}=1$, it follows that $\eta^{m} \in \langle \tau \rangle$, from which we must have that 
 \begin{equation}\label{potencia}
 \begin{array}{l}
 \left(\prod_{j=0}^{m-1}Q(\omega_{m}^{j}x)\right)^{n}=1. 
 \end{array}
 \end{equation}
 
 \begin{claim}\label{claim1}
 Equation \cref{potencia} asserts that $Q(x)$ is either an $nm$-root of unity or it has the form
\[
\begin{array}{l}
Q(x)=\lambda \prod_{u=1}^{A} \frac{x-\alpha_{u}}{x-\omega_{m}^{q_{u}}\alpha_{u}},
\end{array}
\]
 where $\lambda^{nm}=1$ and $q_{u} \in \{1,\ldots,m-1\}$. 
 \end{claim}
 \begin{proof}
If we write
 $$\begin{array}{l}
 Q(x)=\lambda \frac{\prod_{u=1}^{A} (x-\alpha_{u})}{\prod_{v=1}^{B}(x-\beta_{v})},
 \end{array}
 $$
 then 
 $$\begin{array}{l}
 \prod_{j=0}^{m-1} Q(\omega_{m}^{j}x)=\lambda^{m} \prod_{j=0}^{m-1} \omega_{m}^{(A-B)j} \frac{\prod_{u=1}^{A} (x-\omega_{m}^{m-j}\alpha_{u})}{\prod_{v=1}^{B}(x-\omega_{m}^{m-j}\beta_{v})}=
 \end{array}$$
 $$\begin{array}{l}
 =\lambda^{m} \omega_{m}^{(A-B)m(m-1)/2} \frac{\prod_{u=1}^{A} (x^{m}-\alpha_{u}^{m})}{\prod_{v=1}^{B}(x^{m}-\beta_{v}^{m})}.
 \end{array}$$
 Equation \cref{potencia} asserts that 
 $$\begin{array}{l}
 A=B, \quad 
 \prod_{v=1}^{B}(x^{m}-\beta_{v}^{m})=\lambda^{nm} \prod_{u=1}^{A} (x^{m}-\alpha_{u}^{m}). 
 \end{array}
 $$
 
 So, $\lambda^{nm}=1$ and, up to a permutation of indices, we may assume $\alpha_{u}^{m}=\beta_{u}^{m}$, for $u=1,\ldots, A$.
\end{proof}

By \cref{claim1}, either $l_{j,i}-l_{j,i+1}=0$ or $\omega_{m}^{i-1}q_{j}$ must be either a zero or a pole of order $n$ of the left side of \cref{19}, that is, each $l_{j,i}-l_{j,i+1} \in \{0,\pm n\}$. As $l_{j,i} \in \{1,\ldots, n-1\}$, it follows that $l_{j,1}=\cdots=l_{j,m}$.

\subsection{}
In the other direction, let us assume that $l_{j,1}=\cdots=l_{j,m}=l_{j}$, for every $j=1,\ldots,L$. In this case, $\X$ has equation
\begin{equation}
\begin{array}{l}
y^{n}=x^{\alpha} \prod_{j=1}^{L}(x^{m}-q_{j}^{m})^{l_{j}}.
\end{array}
\end{equation}

A lifting of $\theta(\eta)$ under $\pi(x,y)=x$ is of the form $\widehat{\eta}(x,y)=(\omega_{m}x,\omega_{m}^{\alpha/n}y)$. This asserts that $\eta=\widehat{\eta} \tau^{k}$, for some $k \in \{0,\ldots,n-1\}$, i.e., $\eta(x,y)=(\omega_{m}x,\omega_{n}^{k} \omega_{m}^{\alpha/n} y)$, that is $l=1$.

 \subsection{A consequence}\label{Sec:formacanonica}
The above permits us to observe that, if $\tau$ is a generalized superelliptic automorphism of level $n$, and the reduced group $\overline{N}$ admits an element of order $m$, 
then 
$\X$ can be represented by a cyclic $n$-gonal curve of the form 
\begin{equation}
\begin{array}{l}
\X: \quad y^{n}=x^{l_{0}} (x^{m}-1)^{l_{1}}\prod_{j=2}^{L}(x^{m}-a_{j}^{m})^{l_{j}},
\end{array}
\end{equation}
where any one of the following Harvey's conditions is satisfied:
\begin{enumerate}
\item if $l_{0}=0$, then $m(l_{1}+\cdots+l_{L}) \equiv 0 \;{\rm mod}(n)$ and $\gcd(n,l_{1},\ldots,l_{L})=1$; or
\item if $l_{0} \neq 0$, then 
$\gcd(n,l_{0},l_{1},\ldots,l_{L})=1$.
\end{enumerate}

Note that, in (2) above, either: 
(2.1) $l_{0}+m(l_{1}+\cdots+l_{L}) \equiv 0 \;{\rm mod}(n)$ in case $\infty$ is not a branch value, or
(2.2) $l_{0}+m(l_{1}+\cdots+l_{L}) \notequiv 0 \;{\rm mod}(n)$ in case $\infty$ is a branch value.

\section{Proof of Theorem \ref{unicidad0} and Corollary \ref{unicidad}} \label{Sec:unico}
Let us assume $\X$ admits two superelliptic automorphisms $\tau$ and $\eta$, both of level $n$, that is, each one being central in $G=\Aut (\X)$.
Let $H=\langle \tau \rangle$ and the reduced group $\G=G/H$.
We proceed to investigate when it is possible to have that $\eta \not\in H$.

\subsection{Proof of Theorem \ref{unicidad0}}
As the case $n=2$ corresponds to the hyperelliptic situation, and the hyperelliptic involution is unique, necessarily $n \geq 3$. 

\begin{prop}\label{propo1}
If $\G$ is either trivial, a dihedral group of order not divisible by $4$ or $A_{4}$ or $S_{4}$ or $A_{5}$, then $\eta \in H$.
\end{prop}

\begin{proof}
Assume, to the contrary, that $\eta \not\in H$. Then $\eta$ induces a non-trivial central element of the reduced group $\G$.
As the Platonic groups and the dihedral groups of order not divisible by $4$, have no nontrivial central element, this is a contradiction.
\end{proof}

Let us assume that  $\eta \not\in H$. So, by the above, $n \geq 3$ and 
$\G$ is either a non-trivial cyclic group or a dihedral group of order a multiple of $4$. 
Let us consider, as before, the canonical quotient homomorphism $\theta:G \to \G$,  and let $\pi:\X \to \widehat\C$ be a Galois branched cover with deck group $H$. As $\tau$ is central,
$K=\langle \tau, \eta\rangle<G$ is an abelian group and $\overline{K}=K/H =\langle \theta(\eta) \rangle \cong C_{m}$, where $n=md$ and $m \geq 2$. Since $\theta(\eta)$ has order $m$,  $\eta^{m} \in H$ and it has order $d$. So, replacing $\tau$ by a suitable power (still being a generator of $H$) we may assume that $\eta^{m}=\tau^{m}$.  Now, as noted in Section \ref{Sec:formacanonica},
 we may assume $\X$ to be represented by a cyclic $n$-gonal curve of the form 
\begin{equation}
\begin{array}{l}
\X: \quad y^{n}=x^{l_{0}} (x^{m}-1)^{l_{1}}\prod_{j=2}^{L}(x^{m}-a_{j}^{m})^{l_{j}},
\end{array}
\end{equation}
where one of the following Harvey's conditions is satisfied:
\begin{enumerate}
\item[(C1)] $l_{0}=0$, $m(l_{1}+\cdots+l_{L}) \equiv 0 \mod(n)$ and $\gcd(n,l_{1},\ldots,l_{L})=1$; or
\item[(C2)] $l_{0} \neq 0$ and $\gcd(n,l_{0},l_{1},\ldots,l_{L})=1$.
\end{enumerate}

In this algebraic model, $\tau(x,y)=(x,\omega_{n}y)$, $\pi(x,y)=x$ and $\theta(\eta)(x)=\omega_{m}x$ (where $\omega_{t}=e^{2 \pi i/t}$). In this way, $\eta(x,y)=(\omega_{m}x,\omega_{m}^{l_{0}/n}y)$. Since 
$\eta^{m}=\tau^{m}$ and $\eta$ has order $n$, we may assume the following
\begin{equation}
\left\{ \begin{array}{lll}
{\rm if } \; l_{0} \neq 0: & \eta(x,y)=(\omega_{m}x, \omega_{n}y) & \mbox{and } \; l_{0}=m, \\
{\rm if } \; l_{0}=0: & \eta(x,y)=(\omega_{m}x,y) & \mbox{and } \; n=m.
\end{array}
\right.
\end{equation}

\medskip

\noindent  \noindent  \textbf{i):}    \emph{Case $l_{0}=m$.} In this case,  $\eta(x,y)=(\omega_{m}x, \omega_{n}y)$ and we are in case (C2) above.
The $\eta$-invariant algebra $\C[x,y]^{\langle \eta\rangle}$ is generated by the monomials $u=x^{m}, v=y^{n}$ together those of the form $x^{a}y^{b}$, where $a \in \{0,1,\ldots,m-1\}$ and $b \in \{0,1,\ldots,n-1\}$ (where the case $a=b=0$ is not considered) satisfy that $a+b/d \equiv 0 \mod(m)$. In particular, $b=dr$ for $r \in \{0,1,\ldots,[(n-1)/d]\}$ so that $a+r \equiv 0 \mod(m)$. As $0 \leq a+r \leq (m-1)+[(n-1)/d] \leq (m-1)+[(md-1)/d]<2m$, it follows that $a+r \in \{0,m\}$. As the case $a+r=0$ asserts that $a=b=0$, which is not considered, we must have $a+r=m$, from which we see that the other generators are
given by $t_{1},\ldots, t_{m-1}$, where $t_{j}=x^{m-j}y^{dj}$. As consequence of invariant theory, the quotient curve $\X/\langle \eta\rangle$ corresponds to the algebraic curve
\begin{equation}
{\mathcal Y}: \quad \left\{ \begin{array}{lcl}
t_{1}^{m}&=&u^{m-1}v,\\
t_{2}^{m}&=&u^{m-2}v^{2},\\
&\vdots&\\
t_{m-1}^{m}&=&uv^{m-1},\\
v&=&u (u-1)^{l_{1}}\prod_{j=2}^{L}(u-a_{j}^{m})^{l_{j}}.
\end{array}
\right.
\end{equation}

The curve ${\mathcal Y}$ admits the automorphisms $T_{1},\ldots, T_{m-1}$, where $T_{j}$ is just an amplification of the $t_{j}$-coordinate by $\omega_{m}$ and acts as the identity on all the other coordinates. The group generated by all of these automorphisms is 
\begin{equation}
{\mathcal U}=\langle T_{1},\ldots,T_{m-1}\rangle \cong C_{m}^{m-1}.
\end{equation} 
The  Galois branched cover map $\pi_{\mathcal U}:{\mathcal Y} \to \widehat\C: (u,v,t_{1},\ldots,t_{m-1}) \mapsto u$ has ${\mathcal U}$ as its deck group. Let us observe that the values $0$, $a_{1}^{m},\ldots, a_{L}^{m}$ belongs to the branch set of $\pi_{\mathcal U}$.
Since ${\mathcal Y}=\X/\langle \eta \rangle$ has genus zero and the finite abelian groups of automorphisms of the Riemann sphere are either the trivial group, a cyclic group or $V_{4}=C_{2}^{2}$, the group ${\mathcal U}$ is one of these three types.
As $m \geq 2$, the group ${\mathcal U}$ cannot be the trivial group nor can it be isomorphic to the Klein group $V_{4}=C_{2}^{2}$. It follows that  
${\mathcal U}$ is a cyclic group; so $m=2$ and,  in particular, $n=2d$, where $d \geq 2$, and 
\begin{equation}\label{dd}
\begin{array}{l}
\X: \quad y^{2d}=x^{2} (x^{2}-1)^{l_{1}} \prod_{j=2}^{L}(x^{2}-a_{j}^{2})^{l_{j}}.
\end{array}
\end{equation}

Harvey's condition (a) is equivalent to have $\gcd(2d,2,l_{1},\ldots,l_{L})=1$, which is satisfied if some of the exponents $l_{j}$ is odd. Without loss of generality, we may assume that $l_{1}$ is odd.   In this case the curve ${\mathcal Y}$ is given by 
\begin{equation}
{\mathcal Y}: \quad \left\{ \begin{array}{lcl}
t_{1}^{2}&=& uv,\\
v&=&  u(u-1)^{l_{1}}\prod_{j=2}^{L}(u-a_{j}^{2})^{l_{j}},
\end{array}
\right.
\end{equation}
which is isomorphic to the curve 
\begin{equation}\label{cero}
\begin{array}{l}
w^{2}=(u-1)^{l_{1}}\prod_{j=2}^{L}(u-a_{j}^{2})^{l_{j}}.
\end{array}
\end{equation}

As this curve must have genus zero, and $l_{1}$ is odd, the number of indices $j \in \{2,\ldots,L\}$ for which $l_{j}$ is odd must be at most one.

\begin{enumerate}
\item[(i)] If $l_{1}$ is the only odd exponent and $l_{j}=2\widehat{l}_{j}$, for $j=2,\ldots,L$, then the condition 
$\gcd(2d,2,l_{1},2\widehat{l}_{2},\ldots,2\widehat{l}_{L})=1$ is equivalent to
$\gcd(d,l_{1},\widehat{l}_{2},\ldots,\widehat{l}_{L})=1$.

\item[(ii)] If there are exactly two of the exponents being odd, then we may assume, without loss of generality, that $l_{1}$ and $l_{2}$ are the only odd exponents. This means that the curve
in \cref{cero} is isomorphic to $\widehat{w}^{2}=(u-1)^{l_{1}}(u-a_{2}^{2})^{l_{2}}$, where $\widehat{w}=w/\prod_{j=3}^{L}(u-a_{j})^{l_{j}/2}$.
If we write $l_{j}=2\widehat{l}_{j}$, for $j=3,\ldots,L$, then the condition 
$\gcd(2d,2,l_{1},l_{2},2\widehat{l}_{3},\ldots,2\widehat{l}_{L})=1$ is equivalent to 
$\gcd(d,l_{1},l_{2},\widehat{l}_{3},\ldots,\widehat{l}_{L})=1$.
\end{enumerate}

\medskip
\noindent  \textbf{ii):}    \emph{Case $l_{0}=0$.} In this case, $m=n$, $\eta(x,y)=(\omega_{n}x, y)$ and we are in case (C1) above.
The $\eta$-invariants algebra $\C[x,y]^{\langle \eta\rangle}$ is generated by the monomials $u=x^{n}, v=y$. As a consequence of the invariant theory, the quotient curve $\X/\langle \eta\rangle$ corresponds to one of the following algebraic curves 
\begin{equation}
{\mathcal Y}_{1}: \quad \left\{ \begin{array}{lcl}
v^{n}&=&(u-1)^{l_{1}}
\end{array}
\right.
\end{equation}
or 
\begin{equation}
{\mathcal Y}_{2}: \quad \left\{ \begin{array}{lcl}
v^{n}&=&(u-1)^{l_{1}}\prod_{j=2}^{L}(u-a_{j}^{n})^{l_{j}}.
\end{array}
\right.
\end{equation}

As 
${\mathcal Y}$ must have genus zero and $n \geq 3$, we should have either ${\mathcal Y}_{1}$ or ${\mathcal Y}_{2}$ with $L=2$ and $l_{1}+l_{2} \equiv 0 \mod(n)$.
In particular, we have one of the two cases below for $\X$:
\begin{equation}
\begin{array}{ll}
(1) & \X: \quad y^{n}=(x^{n}-1)^{l_{1}}.\\
(2)& \X: \quad y^{n}=(x^{n}-1)^{l_{1}}(x^{n}-a_{2}^{n})^{l_{2}}, \; l_{1}+l_{2} \equiv 0 \mod(n).
\end{array}
\end{equation}

Note that, for situation (1) above,  we may assume $l_{1}=1$ (this is the classical Fermat curve of degree $n$). As the group of automorphisms of the classical Fermat curve of degree $n$ is $C_{n}^{2} \rtimes S_{3}$, we may see that $\tau$ is not central; that is, it is not a generalized superelliptic Riemann surface of level $n$. In case (2), Harvey's conditions hold exactly when $\gcd(n,l_{1},l_{2})=1$. As $l_{1}+l_{2} \equiv 0 \mod(n)$ and $l_{1},l_{2} \in \{1,\ldots,n-1\}$, we have that $l_{1}+l_{2}=n$. 
If we write $l_{2}=n-l_{1}$, then 
\begin{equation}
\begin{array}{l}
\left(\frac{x^{n}-1}{x^{n}-a_{2}^{n}}\right)^{l_{1}}=\frac{y^{n}}{(x^{n}-a_{2}^{n})^{n}},
\end{array}
\end{equation}
and by writing $l_{1}=n-l_{2}$ we also have that
\begin{equation}
\begin{array}{l}
\left(\frac{x^{n}-a_{2}^{n}}{x^{n}-1}\right)^{l_{2}}=\frac{y^{n}}{(x^{n}-1)^{n}}.
\end{array}
\end{equation}
Then the M\"obius transformation $M(x)=a_{2}/x$ induces the automorphism
\begin{equation}
\begin{array}{l}
\alpha(x,y)=\left(\omega_{n}\frac{a_{2}}{x},\frac{-a_{2}^{l_{2}} (x^{n}-1)(x^{n}-a_{2}^{n})}{x^{n} y} \right),
\end{array}
\end{equation}
which does not commute with $\eta(x,y)=(\omega_{n}x,y)$ since $n \geq 3$, a contradiction.

\subsection{Proof of Corollary \ref{unicidad}}
Let $\X$ be a a cyclic $n$-gonal Riemann surface admitting superelliptic automorphisms $\tau, \eta$, both of level $n$, such that $\langle \tau \rangle \neq \langle \eta \rangle$.  By Theorem \ref{unicidad0}, $\X$ has an equation of the form as in \cref{dd}, where $n=2d \geq 4$. The factor $x^{2}$ in such an equation asserts that $0$ is a branch value of order $d=n/2$.

\section{A remark on the field of moduli of superelliptic curves}\label{Sec:moduli}
\subsection{Field of definitions and the field of moduli}
As a consequence of the Riemann-Roch theorem, every closed Riemann surface $\X$ can be described as a complex projective irreducible algebraic curve, say defined as the common zeros of the homogeneous polynomials $P_{1},\ldots, P_{r}$.    If $\sigma \in \Gal (\C)$, the group of field automorphisms of $\C$,  then $X^{\sigma}$ will denote the curve defined as the common zeros of the polynomials $P_1^{\sigma}, \ldots, P_{r}^{\sigma}$, where $P_{j}^{\sigma}$ is obtained from $P_{j}$ by applying $\sigma$ to its coefficients. The new algebraic curve $\X^{\sigma}$ is again a closed Riemann surface of the same genus. Let us observe that, if $\sigma, \tau \in \Gal (\C)$, then $X^{\sigma \tau}=(X^{\sigma})^{\tau}$ (we multiply the permutations from left to right).
A subfield ${\mathbb L}$ of $\C$ is called a \emph{field of definition} of $\X$ if there is a curve ${\mathcal Y}$, defined over ${\mathbb L}$, which is isomorphic to $\X$ over $\C$.  Weil's descent theorem \cite{Weil} provides sufficient conditions for a given subfield of $\C$ to be a field of definition of $\X$. These conditions hold 
 if $\X$ has no non-trivial automorphisms (a generic situation for $g \geq 3$).

If $G_{\X}$ is the subgroup of $\Gal (\C)$ consisting of those $\sigma$ so that $\X^{\sigma}$ is isomorphic to $\X$, then the fixed field $M_{\X}$ of $G_{\X}$ is called \emph{the field of moduli} of $\X$. 
The notion of the field of moduli was originally introduced by Shimura \cite{Shimura} for the case of abelian varieties and later extended to more general algebraic varieties by Koizumi \cite{Koizumi}. In that same paper, Koizumi observed that: 
(i) $M_{\X}$ is the intersection of all the fields of definition of $\X$, and 
(ii) $\X$ has a field of definition being a finite extension of $M_{\X}$.

There are examples for which the field of moduli is not a field of definition \cites{Earle, Shimura}. 
In \cite{HQ} the following sufficient condition for a surface to be definable over its field of moduli was obtained.

\begin{thm} \label{thm1}
Let $\X$ be a Riemann surface of genus $g \geq 2$ admitting a subgroup $L<\Aut (\X)$ so that $\X/L$ has genus zero. If 
$L$ is unique in $\Aut(\X)$ and the reduced group  $\Aut(\X)/L$ is different from trivial or cyclic, then $\X$ is definable over its field of moduli.
\end{thm}
 
If $\X$ is hyperelliptic and $L$ is the cyclic group generated by the hyperelliptic involution, then the above 
result is due to Huggins \cite{Hu2}.
 
Another sufficient condition on a curve $\X$ to be definable over its field of moduli, which in particular contains the case of quasiplatonic curves, was provided in \cite{AQ}. We say that $\X$ has \emph{odd signature} if $\X/\Aut (\X)$ has genus zero and in its signature one of the cone orders appears an odd number of times.
 
\begin{thm} \label{thm2}
Let $\X$ be a Riemann surface of genus $g \geq 2$. If $\X$ has an odd signature, then it can be defined over its field of moduli.
\end{thm}

\subsection{Minimal fields of definition of superelliptic curves}
Let $\X$ be a superelliptic curve of level $n$ and $H=\langle \tau \rangle \leq \Aut(\X)$ be a superelliptic group of level $n$.
If $\X$ is non-exceptional, then $H$ is unique (\cref{unicidad}). So, if $\Aut(\X)/H$ is different from trivial or cyclic, then $\X$ is definable over its field of moduli by Theorem \ref{thm1}. By Theorem \ref{thm2}, the same result holds if $\X$ has od signature.

At the level of the exceptional ones, we have seen that $H$ is not unique. But, if $\eta$ is another superelliptic automorphism of level $n$, then there is a power of $\eta$ inside $H$. In this case, we have seen that the quotient of $\X$ by the abelian group $K=\langle \tau, \eta \rangle$ has an odd signature. If $\Aut(\X)=K$, then again $\X$ is definable over its field of moduli.

\begin{thm}\label{teounico}
Let $H \cong C_{n}$ be a superelliptic group of a superelliptic curve $\X$. Then $\X$ is definable over its field of moduli if $\X$ is non-exceptional with either (i) ${\rm Aut}(\X)/H$ different from trivial or cyclic or (ii) ${\rm Aut}(\X)/H$ either trivial or cyclic and $\X$ has an odd signature.
\end{thm}

\section{Appendix A: Algebraic equations for generalized superelliptic curves}\label{Sec:Auto}
Let $\X$ be a generalized superelliptic curve of level $n$ and $\tau \in G=\Aut(\X)$ be a generalized superelliptic automorphism of level $n$ (so, it is central in its normalizer $N$). 
We proceed to describe explicit algebraic equations for $\X$ and also explicit generators for $N$, by making a subtle modification of the classical method done by Horiuchi in \cite{Horiuchi} for the hyperelliptic situation.

Let $\pi:\X \to \widehat\C$ be a Galois branched cover 
with deck group $H=\langle \tau \rangle$ and let ${\mathcal B}_{\pi}=\{p_{1},\ldots,p_{s}\} \subset \widehat\C$ be its set of branch values. 
Let $\theta:N \to \N$ be the surjective homomorphism satisfying $\theta(\eta) \circ \pi = \pi \circ \eta$, for every $\eta \in N$. Recall that 
$\N$ is one of the finite subgroups of $\PSL_{2}(\C)$ (as described in \cref{Sec:gruposfinitos}) keeping the set ${\mathcal B}_{\pi}$ invariant. 

Let us consider the Galois branched cover $f=f_{\N}:\widehat\C \to \widehat\C$ with $\N$ as its deck group (as described in \cref{Sec:gruposfinitos}). Let $P(x), Q(x) \in \C[x]$ be relatively prime polynomials such that $f(x)=\frac {P(x)} {Q(x)}$.

The collection ${\mathcal B}_{\pi}$ is $\N$-invariant and, 
by \cref{exponentes}, if for $t \in \N$ it holds that $t(p_{i})=p_{j}$, then $l_{i}=l_{j}$. In particular, we may consider the partition
${\mathcal B}_{\pi}={\mathcal B}_{\pi}^{crit} \cup {\mathcal B}_{\pi}^{*}$, where ${\mathcal B}_{\pi}^{crit}$ consists of those branch values with non-trivial $\N$-stabilizer. For simplicity, we assume $\infty \notin {\mathcal B}_{\pi}^{*}$ (but it might happen that $\infty \in {\mathcal B}_{\pi}^{crit}$).

\subsection{Horiuchi's general process}
\subsubsection{Computing algebraic models}

There is at most $T \leq 3$ disjoint $\N$-orbits of the points in ${\mathcal B}^{crit}_{\pi}$.

If $\N \cong C_{m}$, then $T \leq 2$; each such orbit has cardinality one.

If $\N \cong D_{m}$, then $T \leq 3$; at most one orbit of cardinality $2$ and at most two others, each of cardinality $m$.

If $\N \cong A_{4}$, then $T \leq 3$; at most one orbit of cardinality $6$ and at most two others, each of cardinality $4$.

If $\N \cong S_{4}$, then $T \leq 3$; at most one orbit of cardinality $8$, one of cardinality $6$ and another of cardinality $12$.

If $\N \cong A_{5}$, then $T \leq 3$; at most one orbit of cardinality $20$, one of cardinality $30$ and another of cardinality $12$.

Let us denote these orbits (eliminating $\infty$ from its orbit if it is a branch value of $\pi$) by
\begin{equation}
{\mathcal O}^{crit}_{u}=\{q_{u,1},\ldots,q_{u,s_{u}}\},\; u=1,\ldots, T,
\end{equation}
where $s=s_{1}+\cdots+s_{T}$ is the cardinality of ${\mathcal B}_{\pi}^{crit}$ if $\infty \notin {\mathcal B}_{\pi}^{crit}$ (otherwise, this cardinality is $s+1$).

Similarly,  let the disjoint $\N$-orbits of the points in ${\mathcal B}^{*}_{\pi}$ be given by
\begin{equation}
{\mathcal O}^{*}_{k}=\{p_{k,1},\ldots,p_{k,|\N|}\},\; k=1,\ldots, L,
\end{equation}\
(so, $L|\N|$ is the cardinality of ${\mathcal B}_{\pi}^{*}$).

As, for $k=1,\ldots,L$, 
$$\begin{array}{l}
\prod_{j=1}^{|\N|} (x-p_{k,j})=P(x)-f(p_{k,1})Q(x),
\end{array}
$$
our curve can be written as 
\begin{equation}\label{curvita}
\begin{array}{l}
\X: \quad y^{n}=\prod_{u=1}^{T} R_{u}(x)^{{\widehat l}_{u}}   \prod_{k=1}^{L} \left(P(x)-f(p_{k,1})Q(x)\right)^{\widetilde{l}_{k}},
\end{array}
\end{equation}
where 
\begin{enumerate}
\item $R_{u}(x)=\prod_{j=1}^{s_{u}} (x-q_{u,j})$

\item $\widehat{l}_{u} \in \{0,1,\ldots,n-1\}$ and $\widetilde{l}_{k} \in \{1,\ldots,n-1\}$.

\item $\gcd(n,\widehat{l}_{1},\ldots\widehat{l}_{T},\widetilde{l}_{1},\ldots,\widetilde{l}_{L})=1$ (where we eliminate a zero if appears).
\item if $\infty \notin {\mathcal B}^{crit}_{\pi}$, then
$$\begin{array}{l}
\sum_{u=1}^{T} s_{u}\widehat{l}_{u} + \sum_{k=1}^{L} |\N| \widetilde{l}_{k} \equiv 0 \mod n.
\end{array}
$$

\item If  $\infty \in {\mathcal O}_{v}^{crit}$, then
$$\begin{array}{l}
(1+s_{v}) \widehat{l}_{v}+\sum_{u=1, u \neq v}^{T} s_{u}\widehat{l}_{u} + \sum_{k=1}^{L} |\N| \widetilde{l}_{k} \equiv 0 \mod n.
\end{array}
$$

\end{enumerate}

\subsubsection{Computing the elements of $N$}
Let $\eta \in N$ and $b=\theta(\eta)$. As $\tau$ commutes with $\eta$, by \cref{lemaconmuta}, $\eta(x,y)=(b(x),F(x)y)$, where $F(x) \in \C(x)$. Below, we sketch how to compute such $F(x)$.

\begin{lem}\label{lemita1}
Let ${\mathcal O}=\{a_{1},\ldots,a_{r}\}$ a full $\N$-orbit (in our case, this is one of the ${\mathcal O}^{crit}_{u}$ or ${\mathcal O}^{*}_{k}$). If $b \in \N$, then the following hold.
\begin{enumerate}
\item If $\infty \notin {\mathcal O}$, then
$$\begin{array}{l}
\prod_{j=1}^{r}(b(x)-a_{j})=\left(b'(x)\right)^{r/2} \left(\prod_{j=1}^{r} b'(a_{j})\right)^{1/2} \prod_{j=1}^{r}(x-a_{j}).
\end{array}
$$

\item If $a_{r}=\infty$ and $b(\infty)=\infty$, then
$$\begin{array}{l}
\prod_{j=1}^{r-1}(b(x)-a_{j})=\left(b'(x)\right)^{(r-1)/2} \left(\prod_{j=1}^{r-1} b'(a_{j})\right)^{1/2} \prod_{j=1}^{r-1}(x-a_{j}).
\end{array}
$$

\item If $a_{r}=\infty$, $b(a_{r-1})=\infty$ and $b(\infty)=a_{s}$, where $s \neq r-1$, and $b(a_{t})=a_{r-1}$, then 
$$\begin{array}{l}
\prod_{j=1}^{r-1}(b(x)-a_{j})=
\end{array}$$
$$\begin{array}{l}
\frac{(a_{r-1}-a_{s})^{1/2}(a_{r-1}-a_{t})^{1/2}}{x-a_{r-1}}\left(b'(x)\right)^{(r-1)/2} \left(\prod_{j=1}^{r-2} b'(a_{j})\right)^{1/2} \prod_{j=1}^{r-1}(x-a_{j}).
\end{array}$$

\item If $a_{r}=\infty$, $b(a_{r-1})=\infty$ and $b(\infty)=a_{r-1}$,  then 
$$\begin{array}{l}
\prod_{j=1}^{r-1}(b(x)-a_{j})=-\left(b'(x)\right)^{r/2} \left(\prod_{j=1}^{r-2} b'(a_{j})\right)^{1/2} \prod_{j=1}^{r-1}(x-a_{j}).
\end{array}
$$

\end{enumerate}
\end{lem}

\begin{proof}
The equalites are consequence of the fact that, for $a, b(a) \in \C$,
\[
b(x)-b(a)=b'(x)^{1/2} b'(a)^{1/2} (x-a).
\]
\end{proof}

If, in the above lemma, we replace ${\mathcal O}$ by ${\mathcal O}^{crit}_{u}$, then we obtain an equality
$$\begin{array}{l}
R_{u}(b(x))=\prod_{j=1}^{s_{u}}(b(x)-q_{u,j})=Q_{u}(x) \prod_{j=1}^{s_{u}}(x-q_{u,j})=Q_{u}(x) R_{u}(x).
\end{array}
$$

Similarly, if we replace ${\mathcal O}$ by ${\mathcal O}^{*}_{k}$ and set
$$\begin{array}{l}
S_{k}(x):=P(b(x))-f(p_{k,1})Q(b(x))=\prod_{j=1}^{|\N|}(x-p_{k,j}),
\end{array}
$$
then we obtain an equality
$$\begin{array}{l}
S_{k}(b(x))=\prod_{j=1}^{|\N|}(b(x)-p_{k,j})=L_{k}(x) \prod_{j=1}^{|\N|}(x-p_{k,j})=L_{k}(x) S_{k}(x).
\end{array}
$$

It can be checked, by plugging directly into the equation for $\X$, that
\begin{equation}
\begin{array}{l}
F(x)^{n}=\prod_{u=1}^{T}Q_{u}(x)^{\widehat{l}_{u}} \prod_{k=1}^{L} L_{k}(x)^{\widetilde{l}_{k}}.
\end{array}
\end{equation}

\subsection{Explicit computations}
Below, for each of the possibilities for $\N$, we proceed to explicitly describe the above procedure. 
In the following, if $l_{u}>0$, then we set $n_{u}=\gcd(n,l_{u})$.

\begin{thm}\label{gonalescentral}
Let $\X$ be a generalized superelliptic curve of level $n$, $\tau \in G=\Aut(\X)$ be a generalized superelliptic automorphism of order $n$ and $N$ be the normalizer of $H=\langle \tau \rangle$ in $G$. Then, up to isomorphisms, $\X$, $\tau$ and $N$ are described as indicated in the above cases.

\begin{table}[htp]
\caption{default}
\begin{center}
\begin{tabular}{|c|c|c|}
\hline
$\bar N $  &   Equation & Genus \\
\hline
$C_m$  &  \cref{equCm} & \cref{gen-Cm} \\
$D_m$ & \cref{equDm}  & \cref{gen-Dm}  \\
$A_4$  &  \cref{equA4}  & \cref{gen-A4}  \\
$S_4$ &   \cref{equS4}  & \cref{gen-S4}   \\
$A_5$   & \cref{equA5}  & \cref{gen-A5}   \\
\hline
\end{tabular}
\end{center}
\label{default}
\end{table}%
\end{thm}
\begin{proof} 
We will consider all cases one by one. \\

\noindent \textbf{Case }{$\bf \N \cong C_{m}$:}   In this case, $\N=\big\langle a(x)=\omega_{m} x \big\rangle$ and  the curve $\X$ has the form 
\begin{equation}\label{equCm}
\begin{array}{l}
 \X: \quad y^{n}=x^{l_{0}} (x^{m}-1)^{l_{1}}\prod_{j=2}^{r} (x^{m}-a_{j}^{m})^{l_{j}},
 \end{array}
\end{equation} 
where (i) $a_{2},\ldots,a_{r} \in \C-\{0,1\}, \; a_{i}^{m} \neq a_{j}^{m}$ and (ii) $\gcd(n,l_{0},l_{1},\ldots,l_{r})=1$.  
If 
\[
\alpha(x,y)=(\omega_{m} x, \omega_{m}^{l_{0}/n}y),
\]
 then 
\[
N=\langle \tau, \alpha: \tau^{n}=1, \alpha^{m}=\tau^{l_{0}}, \tau \alpha = \alpha \tau \rangle.
\]

The signature of $\X/H$ is 
\[ 
\left\{ 
\begin{array}{ll}
\left(0;\frac{n}{n_{1}},\stackrel{m}{\ldots},\frac{n}{n_{1}},\ldots, \frac{n}{n_{r}},\stackrel{m}{\ldots},\frac{n}{n_{r}}\right), & \textit{if  } l_{0}=0,\; m\sum_{j=1}^{r}l_{j} \equiv 0 \mod(n),   \\
\left(0;\frac{n}{n_{0}},\frac{n}{n_{1}},\stackrel{m}{\ldots},\frac{n}{n_{1}},\ldots, \frac{n}{n_{r}},\stackrel{m}{\ldots},\frac{n}{n_{r}}\right), & \textit {if }  l_{0} \neq 0,\; l_{0}+m\sum_{j=1}^{r}l_{j} \equiv 0 \mod(n),\\
\left(0;\frac{n}{n_{0}},\frac{n}{n_{r+1}},\frac{n}{n_{1}},\stackrel{m}{\ldots},\frac{n}{n_{1}},\ldots, \frac{n}{n_{r}},\stackrel{m}{\ldots},\frac{n}{n_{r}}\right), & \textit{if  } l_{0} \neq 0, \; l_{0}+m\sum_{j=1}^{r}l_{j} \not\equiv 0 \mod(n),
\end{array}
\right.
\]
where (in the last situation) $l_{r+1} \in \{1,\ldots,n-1\}$ is the class of $-(l_{0}+m\sum_{j=1}^{r}l_{j})$ module $n$.
The signature of $\X/N$ is 
\[\left\{ 
\begin{array}{ll}
\left(0;m,m,\frac{n}{n_{1}},\frac{n}{n_{2}},\ldots,\frac{n}{n_{r}}\right), & \mbox{if $l_{0}=0$,\; $m\sum_{j=1}^{r}l_{j} \equiv 0 \mod(n)$, }\\
\left(0;m,\frac{mn}{n_{0}},\frac{n}{n_{1}},\frac{n}{n_{2}},\ldots,\frac{n}{n_{r}}\right), & \mbox{if $l_{0} \neq 0$, $l_{0}+m\sum_{j=1}^{r}l_{j} \equiv 0 \mod(n)$,}\\
\left(0;\frac{mn}{n_{0}},\frac{mn}{n_{r+1}},\frac{n}{n_{1}},\frac{n}{n_{2}},\ldots,\frac{n}{n_{r}}\right), & \mbox{if $l_{0} \neq 0$, $l_{0}+m\sum_{j=1}^{r}l_{j} \not\equiv 0 \mod(n)$,}
\end{array}
\right.
\]

The genus of $\X$ is
\begin{equation}\label{gen-Cm}
\left\{ 
\begin{array}{l}
1+\frac{1}{2} \left(   (rm-2)n-m\sum_{j=1}^{r}n_{j}  \right), \quad  \text{if }     l_{0}=0,\; m\sum_{j=1}^{r}     l_{j} \equiv 0 \mod(n),   \\
1+\frac{1}{2}\left((rm-1)n-m\sum_{j=1}^{r}n_{j}\right),  \quad \text{if }    l_{0} \neq 0, l_{0}+m\sum_{j=1}^{r}l_{j} \equiv 0 \mod(n),     \\
1+\frac{1}{2}\left(rmn-m\sum_{j=1}^{r}n_{j}\right),  \quad  \text{if }   l_{0} \neq 0, l_{0}+m\sum_{j=1}^{r}l_{j} \not\equiv 0 \mod(n).     \\
\end{array}
\right.
\end{equation}

\noindent \textbf{Case }{$\bf \N \cong D_{m}$:}  In this case, $D_{m}:=\Big\langle a(x)=\omega_{m}x, b(x)=\frac{1}{x}\Big\rangle$ and   the curve $\X$ has the form
\begin{equation}\label{equDm}  
\begin{array}{l}
 \X: \quad y^{n}=x^{l_{0}}(x^{m}-1)^{l_{r+1}}(x^{m}+1)^{l_{r+2}} \prod_{j=1}^{r} (x^{2m}-(a_{j}^{m}+a_{j}^{-m})x^{m}+1)^{l_{j}},
\end{array}
\end{equation}
where (i) $a_{i}^{\pm m} \neq a_{j}^{\pm m} \neq 0,\pm1$, (ii) $2l_{0}+m(l_{r+1}+l_{r+2})+2m(l_{1}+\cdots+l_{r}) \equiv 0 \mod(n)$ and (iii)
$\gcd(n,l_{0},l_{1},\ldots,l_{r+2})=1$.
If $\alpha$ and $\beta$ are as follows 
\[
\alpha(x,y)  =    (\omega_{m} x, \omega_{m}^{l_{0}/n}y), \quad 
\beta(x,y)  =      \left(  \frac{1}{x}, \frac{(-1)^{l_{r+1}/n}}   {x^{(2l_{0}+m(l_{r+1}+l_{r+2}+2(l_{1}+\cdots+l_{r})))/n}}      y \right),
\]
then
\[
N=\langle \tau,\alpha, \beta: \tau^{n}=1, \alpha^{m}=\tau^{l_{0}}, \; \beta^{2}=\tau^{l_{r+1}}, \; \tau \alpha=\alpha \tau,\; \tau \beta=\beta \tau \rangle.
\]

The signature of $\X/H$ is
$$\begin{array}{l}
 \left(0;\frac{n}{n_{0}},\frac{n}{n_{0}},\frac{n}{n_{r+1}},\stackrel{m}{\ldots},\frac{n}{n_{r+1}},\frac{n}{n_{r+2}},\stackrel{m}{\ldots},\frac{n}{n_{r+2}},
\frac{n}{n_{1}},\stackrel{2m}{\ldots},\frac{n}{n_{1}},\ldots, \frac{n}{n_{r}},\stackrel{2m}{\ldots},\frac{n}{n_{r}}\right),
\end{array}
$$
the signature of $\X/N$ is
$$\begin{array}{l}
\left(0;\frac{mn}{n_{0}},\frac{2n}{n_{r+1}},\frac{2n}{n_{r+2}}, \frac{n}{n_{1}},\frac{n}{n_{2}},\ldots, \frac{n}{n_{r}}\right),
\end{array}
$$
and the genus of $\X$ is
\begin{equation}\label{gen-Dm}
\begin{array}{l}
g=1+\frac{1}{2}\left(2m(r+1)n-2n_{0}-m\left(n_{r+1}+n_{r+2}+2\sum_{j=1}^{r}n_{j}\right)\right).
\end{array}
\end{equation}

\noindent  \textbf{Case }{$\bf \N \cong A_{4}$:}   In this case, $A_{4}:=\Big\langle a(x)=-x, b(x)=\frac{i-x}{i+x}\Big\rangle$ and 
$\X$ has the form 
\begin{equation}\label{equA4}
\begin{array}{l}
\X: \;  y^{n}=   R_{1}(x)^{l_{r+1}} R_{2}(x)^{l_{r+2}} R_{3}(x)^{l_{r+3}} \prod_{j=1}^{r} \left(R_{1}(x)^{3} + 12 i b_{j} \sqrt{3} R_{3}(x)^{2}\right)^{l_{j}},
\end{array}
\end{equation}
where
$$
\begin{array}{ll}
R_{1}(x) & =x^{4}-2i\sqrt{3} x^{2}+1, \\
R_{2}(x) & =x^{4}+2i\sqrt{3} x^{2}+1, \\
R_{3}(x) & =x(x^{4}-1), \\
 f(x) & =\frac{R_{1}(x)^{3}}{-12 i \, \sqrt{3} \; R_{3}(x)^{2}}, \\
\end{array}
$$
such that (i) $b_{j} \neq b_{i} \in \C\setminus\{0,1\}$, 
(ii) $4(l_{r+1}+l_{r+2})+6l_{r+3}+12(l_{1}+\cdots+l_{r})\equiv 0 \mod(n)$, and  (iii) $\gcd(n,l_{1},\ldots,l_{r+3})=1$. If
$$\begin{array}{l}
\alpha(x,y)=(-x,(-1)^{l_{r+3}/n}y), \quad \beta(x,y)=(b(x),F(x)y),
\end{array}
$$
where
$$\begin{array}{l}
F(x)= \frac{ 2^{(l_{r+1}+l_{r+2})/n}(1- I\sqrt{3})^{l_{r+1}/n} (1+ I\sqrt{3})^{l_{r+2}/n} (8i)^{l_{r+3}/n}  (-64)^{(l_{1} + \cdots + l_{r})/n } }    
{(x+i)^{(4(l_{r+1}+l_{r+2})+6l_{r+3}+12(l_{1}+\cdots+l_{r}))/n}},
\end{array}
$$
then 
$$
\begin{array}{ll}
N  = &  \langle \tau, \alpha, \beta:  \, \, \tau^{n}=1,\alpha^{2}=\tau^{l_{r+3}}, \; \beta^{3}=\tau^{l_{r+1}+l_{r+2}+l_{r+3}+l_{1}+\cdots+l_{r}},\\
&(\alpha \beta)^{3}=\tau^{l_{r+1}+l_{r+2}+3l_{r+3}+l_{1}+\cdots+l_{r}},   \tau \alpha=\alpha \tau, \;  \tau \beta=\beta \tau\rangle.
 \end{array}
$$

The signature of $\X/H$ is
$$\begin{array}{l}
\left(0;\frac{n}{n_{r+1}},\stackrel{4}{\ldots},\frac{n}{n_{r+1}}, \frac{n}{n_{r+2}},\stackrel{4}{\ldots},\frac{n}{n_{r+2}} ,\frac{n}{n_{r+3}},\stackrel{6}{\ldots},\frac{n}{n_{r+3}}, \frac{n}{n_{1}},\stackrel{12}{\ldots},\frac{n}{n_{1}},\ldots, \frac{n}{n_{r}},\stackrel{12}{\ldots},\frac{n}{n_{r}}\right),
\end{array}
$$
the signature of $\X/N$ is
$$\begin{array}{l}
\left(0;\frac{3n}{n_{r+1}},\frac{3n}{n_{r+2}},\frac{2n}{n_{r+3}}, \frac{n}{n_{1}},\frac{n}{n_{2}},\ldots,\frac{n}{n_{r}}\right),
\end{array}
$$
and the genus of $\X$ is  
\begin{equation}\label{gen-A4}
\begin{array}{l}
g=1-n+2n_{r+1}+2n_{r+2}+3n_{r+2}-6\sum_{j=1}^{r}n_{j}.
\end{array}
\end{equation} 

\noindent  \textbf{Case }  {$\bf \N \cong S_{4}$:}
In this case, $S_{4}:=\Big\langle a(x)=ix, b(x)=\frac{i-x}{i+x}\Big\rangle$ and $\X$ has the form 
\begin{equation}\label{equS4}
\begin{array}{l}
\X: \quad y^{n}=R_{1}(x)^{l_{r+1}} R_{2}(x)^{l_{r+2}}R_{3}(x)^{l_{r+3}} \prod_{j=1}^{r} (R_{1}(x)^{3}-108 b_{j} R_{3}(x)^{4})^{l_{j}},
\end{array}
\end{equation}
where
$$
\begin{array}{ll}
R_{1}(x)		& =x^{8}+14x^{4}+1, \\
 R_{2}(x)		& =x^{12}-33x^{8}-33x^{4}+1, \\
R_{3}(x)		& =x(x^{4}-1),    \\
f(x)			& =\frac{R_{1}(x)^{3}}{108R_{3}(x)^{4}},
\end{array}
$$
such that (i) $b_{j} \neq b_{i} \in \C \setminus\{0,1\}$, (ii) $8l_{r+1}+12l_{r+2}+6l_{r+3}+24(l_{1}+\cdots+l_{r}) \equiv 0 \mod(n)$ and 
(iii) $\gcd(n,l_{1},\ldots,l_{r+3})=1$. If
$$\begin{array}{l}
\alpha(x,y)=(ix,i^{l_{r+3}/n}y),  \;\beta(x,y)  =(b(x),F(x)y),
\end{array}
$$
$$\begin{array}{l}
F(x)  = \frac   {16^{l_{r+1}/n} \cdot (-64)^{l_{r+2}/n} \cdot (8i)^{l_{r+3}/n} \cdot 4096^{(l_{1}+\cdots+l_{r}))/n}}  
{(x+i)^{(8l_{r+1}+12l_{r+2}+6l_{r+3}+24(l_{1}+\cdots+l_{r}))/n}},
\end{array}
$$
then
$$
\begin{array}{ll}
N=&\langle \tau, \alpha, \beta: \;    \tau^{n}=1, \; \alpha^{4}=\tau^{l_{r+3}}, \;  \beta^{3}=\tau^{l_{r+1}+l_{r+2}+l_{r+3}+l_{1}+\cdots+l_{r}}, \\
& (\alpha \beta)^2=\tau^{}, \; \tau \alpha=\alpha \tau, \; \tau \beta= \beta \tau  \rangle. 
\end{array}
$$

The signature of $\X/H$ is
$$\begin{array}{l} 
\left(0;\frac{n}{n_{r+1}},\stackrel{8}{\ldots},\frac{n}{n_{r+1}},\frac{n}{n_{r+2}},\stackrel{12}{\ldots},\frac{n}{n_{r+2}}, \frac{n}{n_{r+3}},\stackrel{6}{\ldots},\frac{n}{n_{r+3}},\frac{n}{n_{1}},\stackrel{24}{\ldots},\frac{n}{n_{1}},\ldots, \frac{n}{n_{r}},\stackrel{24}{\ldots},\frac{n}{n_{r}}\right), 
\end{array}
$$
the signature of $\X/N$ is
$$\begin{array}{l}
\left(0;\frac{3n}{n_{r+1}},\frac{2n}{n_{r+2}},\frac{4n}{n_{r+3}},\frac{n}{n_{1}},\frac{n}{n_{2}},\ldots,\frac{n}{n_{r}}\right), 
\end{array}
$$
and the genus of $\X$ is
\begin{equation}\label{gen-S4}
\begin{array}{l}
  g=1+12(1+r)n-4n_{r+1}-6n_{r+2}-3n_{r+3}-12\sum_{j=1}^{r}n_{j}. 
  \end{array}
\end{equation}  

\noindent \textbf{Case }  {$\bf \N \cong A_{5}$:}
In this case, 
\[
A_{5}:=\Big\langle a(x)=\omega_{5} x, b(x)=\frac{(1-\omega_{5}^{4})x+(\omega_{5}^{4}-\omega_{5})}{(\omega_{5}-\omega_{5}^{3})x+(\omega_{5}^{2}-\omega_{5}^{3})}\Big\rangle
\]
and $\X$   has the form
\begin{equation}\label{equA5}
\begin{array}{l}
\X: \quad y^{n}=R_{1}(x)^{l_{r+1}} R_{2}(x)^{l_{r+2}}R_{3}(x)^{l_{r+3}}     \prod_{j=1}^{r} (R_{1}(x)^{3}- 1728  b_{j} R_{3}(x)^{5})^{l_{j}},
\end{array}
\end{equation}
where
$$
\begin{array}{ll}
R_{1}(x) 		& = - x^{20}+228x^{15}-494x^{10}-228x^{5}-1, \\
 R_{2}(x) 		& =  x^{30}+522 x^{25}-10005 x^{20}-10005 x^{10}-522 x^5+1, \\
  R_{3}(x)	 	& =  x(x^{10}+11x^{5}-1), \\
  f(x) 			& =  \frac{ R_{1}(x)^{3}}{1728  R_{3}(x)^{5}},\\
\end{array}
$$
such that 
\begin{enumerate}[\upshape(i), nolistsep]
\item $b_{j} \neq b_{i} \in \C  \setminus \{0,1\}$, 
\item   $20l_{r+1}+30l_{r+2}+12l_{r+3}+60(l_{1}+\cdots+l_{r}) \equiv 0 \mod(n),$
\item $\gcd(n,l_{1},\ldots,l_{r+3})=1$.
\end{enumerate}
$$\begin{array}{l}
N=\langle \tau, \alpha,\beta: \, \,  \alpha^{5}=\tau^{l_{r+3}}, \; \beta^{3}=\tau^{l}, \; (\alpha \beta)^{3}=\tau^{t} \rangle, 
\end{array}
$$
such that 
$\alpha(x,y)=(a(x),{\omega_{5}^{l_{r+3}/n}} y)$ and $\beta(x,y)=(b(x),F(x) y)$, 
where $F(x)$ is a rational map satisfying 
$$\begin{array}{l}
F(b^{2}(x)) \cdot F(b(x)) \cdot F(x)={\omega_{n}^{l}},
\end{array}
$$
for a suitable 
${l \in \{0,\ldots,n-1\}}$, and
$$\begin{array}{l}
F(x)^{n}=T_{1}^{l_{r+1}+3(l_{1}+\cdots +l_{r})}(x) T_{2}^{l_{r+2}}(x) T_{3}^{l_{r+3}}(x),
\end{array}
$$
where $T_{j}(x)=\frac {R_{j}(b(x))} {R_{j}(x)}$, for $j=1,2,3$.

The signature of $\X/H$ is
$$\begin{array}{l}
\left(
0;\frac{n}{n_{r+1}},\stackrel{20}{\ldots},\frac{n}{n_{r+1}},\frac{n}{n_{r+2}},\stackrel{30}{\ldots},\frac{n}{n_{r+2}}, \frac{n}{n_{r+3}},\stackrel{12}{\ldots},\frac{n}{n_{r+3}},\frac{n}{n_{1}},\stackrel{60}{\ldots},\frac{n}{n_{1}},\ldots, \frac{n}{n_{r}},\stackrel{60}{\ldots},\frac{n}{n_{r}}
\right),
\end{array}
$$
and the signature of $\X/N$ is
$$\begin{array}{l}
\left( 0;\frac{3n}{n_{r+1}},\frac{2n}{n_{r+2}},\frac{5n}{n_{r+3}},\frac{n}{n_{1}},\frac{n}{n_{2}},\ldots,\frac{n}{n_{r}} \right),
\end{array}
$$
and the genus of $\X$ is
\begin{equation}\label{gen-A5}
\begin{array}{l}
g=1+30(r+1)n-10n_{r+1}-15n_{r+2}-6n_{r+3}-30\sum_{j=1}^{r}n_{j}.
\end{array}
\end{equation}
\end{proof}

\section{Appendix B: Computing cyclic $n$-gonal curves}
Consider the collection ${\mathcal F}_{g}$ of all the tuples $(n,s;n_{1},\ldots,n_{s})$ satisfying the following Harvey's conditions:
\begin{enumerate}[nolistsep] 
\item $n \geq 2$, $s \geq 3$, $2 \leq n_{1} \leq n_{2} \leq \cdots \leq n_{s} \leq n$;
\item $n_{j}$ is a divisor of $n$, for each $j=1,\ldots,s$;
\item ${\rm lcm}\left(n_{1},\ldots,n_{j-1},n_{j+1},\ldots, n_{s}\right)=n$, for every $j=1,\ldots,s$;
\item if $n$ is even, then $\#\{j\in\{1,\ldots,s\}: n/n_{j} \; \mbox{is odd}\}$ is even;
\item $2(g-1)=n\left(s-2-\sum_{j=1}^{s}n_{j}^{-1}\right)$.
\end{enumerate}

For each tuple $(n,s;n_{1},\ldots,n_{s}) \in {\mathcal F}_{g}$ we consider the collection ${\mathcal F}_{g}(n,s;n_{1},\ldots,n_{s})$ of tuples $(l_{1},\ldots,l_{s})$ so that
\begin{enumerate}[nolistsep] 
\item $l_{1},\ldots, l_{s} \in \{1,\ldots, n-1\}$;
\item $l_{1}+\cdots+l_{s}  \equiv 0 \mod(n)$;
\item $\gcd(n,l_{j})=n/n_{j}$, for each $j=1,\ldots,s$.
\end{enumerate}

Now, for each such tuple $(l_{1},\ldots,l_{s}) \in {\mathcal F}_{g}(n,s;n_{1},\ldots,n_{s})$ we may consider the epimorphism
\begin{equation}
\theta:\Delta=\langle c_{1},\ldots, c_{s}: c_{1}^{n_{1}}=\cdots=c_{s}^{n_{s}}=c_{1}\cdots c_{s}=1\rangle \to C_{n}=\langle \tau \rangle: c_{j} \mapsto \tau^{l_{j}}.
\end{equation}

Our assumptions ensure that the kernel $\Gamma=\ker(\theta)$ is a torsion free normal co-compact Fuchsian subgroup of $\Delta$ with $\X={\mathbb H}/\Gamma$ a closed Riemann surface of genus $g$ admitting a cyclic group $H \cong C_{n}$ as a group of conformal automorphisms with quotient orbifold $\X/H={\mathbb H}/\Delta$; a genus zero orbifold with exactly $s$ cone points of respective orders $n_{1},\ldots,n_{s}$. The surface $\X$ corresponds to a cyclic $n$-gonal curve
\begin{equation}
\begin{array}{l}
C(n,s;l_{1},\ldots,l_{s};p_{1},\ldots,p_{s}): \quad y^{n}=\prod_{j=1}^{s}(x-p_{j})^{l_{j}},
\end{array}
\end{equation}
for  pairwise different values $p_{1},\ldots,p_{s} \in \C$, 
and $H$  generated by $\tau(x,y)=(x,\omega_{n}y)$.

Different tuples $(l_{1},\ldots,l_{s}),(l_{1}',\ldots,l_{s}') \in {\mathcal F}_{g}(n,s;n_{1},\ldots,n_{s})$ might provide isomorphic pairs $(\X,H)$ and $(\X',H')$ (i.e., there is an isomorphism between the Riemann surfaces conjugating the cyclic groups). In general this is a difficult problem to determine if different tuples define  isomorphic pairs. But, in the non-exceptional fully generalized superelliptic situation (see \cref{unicidad0}) the uniqueness of the superelliptic cyclic group of level $n$  permits us to see that $(\X,H)$ and $(\X',H')$ are isomorphic pairs if and only if the corresponding 
curves $C(n,s;l_{1},\ldots,l_{s};p_{1},\ldots,p_{s})$ and $C(n,s;l'_{1},\ldots,l'_{s};p'_{1},\ldots,p'_{s})$ are isomorphic, this last being equivalent to the existence of M\"obius transformation $t \in \PSL_{2}(\C)$, a permutation $\eta \in S_{s}$ and an element $u \in \{1,\ldots,n-1\}$ with $\gcd(u,n)=1$,
such that
\begin{enumerate}[nolistsep] 
\item[(a)] $l'_{j} \equiv u l_{\eta(j)} \mod(n)$, for $j=1,\ldots,s$,
\item[(b)] $p'_{\eta(j)}=t(p_{j})$, for $j=1,\ldots,s$.
\end{enumerate}

The above (together with  \cref{exponentes}) may be used to construct all the possible (generalized) superelliptic curves of lower genus  in a similar fashion as done in \cite{MPRZ} for the superelliptic case.

\bibliography{references}{}
\end{document}